\documentclass[12pt]{article}

\usepackage[english]{babel}
\usepackage[utf8]{inputenc}
\usepackage[T1]{fontenc}

\usepackage{amsmath, amsthm, amssymb}

\usepackage[shortlabels]{enumitem}

\usepackage{geometry}
\AtBeginDocument{\def\MR#1{}}

\usepackage{hyperref}

\providecommand{\keywords}[1]{\textbf{Keywords.} #1}
\providecommand{\MSC}[1]{\textbf{AMS Subject Classifications.} #1}

\newtheorem{theorem}{Theorem}[section]
\newtheorem{corollary}[theorem]{Corollary}
\newtheorem{lemma}[theorem]{Lemma}
\newtheorem{proposition}[theorem]{Proposition}
\newtheorem{example}[theorem]{Example}

\theoremstyle{definition}
\newtheorem{definition}[theorem]{Definition}
\newtheorem{remark}[theorem]{Remark}

\renewcommand{\epsilon}{\varepsilon}

\usepackage{color}

\usepackage{xspace}

\newcommand{\C}{\field{C}\xspace}
\newcommand{\R}{\field{R}\xspace}

\newcommand{\Herm}{\field{H}\xspace}

\newcommand{\field}[1]{\ensuremath{\mathbb{#1}}}

\newcommand{\ens}[1]{ \left\{#1\right\} }

\newcommand\loc{{\mathrm{loc}}}

\newcommand\abs[1]{\left|#1\right|}
\def\norm#1{\left\|#1\right\|}

\newcommand\Id{\mathrm{Id}}

\newcommand\clos[1]{\overline{#1}}

\newcommand\Trace{\mathrm{Trace} \,}
\newcommand\MA{\mathrm{MA}}
\newcommand\PSH{\mathrm{PSH}}
\newcommand\real{\mathfrak{Re} \,}
\newcommand\imag{\mathfrak{Im} \,}

\newcommand\diam{\mathrm{diam} \,}

\newcommand\ddc{dd^c \,}
\newcommand\dV{dV}

\newcommand\DD{D^2}
\newcommand\coneb{\mathcal{C}}

\DeclareMathOperator{\la}{\lambda}
\def\Om{\Omega}
\def\B{B_1}
\def\aij{a^{i\bar{j}}}

\title{Uniform estimates for concave homogeneous complex degenerate elliptic equations comparable to the Monge-Ampère equation}

\AtBeginDocument{\def\MR#1{}}

\author{
Soufian Abja\thanks{Institute of Mathematics, Jagiellonian University, Lojasiewicza 6,
30-348 Krakow, Poland (\texttt{Soufian.Abja@im.uj.edu.pl}, \texttt{Slawomir.Dinew@im.uj.edu.pl}, \texttt{math.golive@gmail.com}).}
\and
S{\l}awomir Dinew\footnotemark[1]
\and
Guillaume Olive\footnotemark[1]
}

\date{\today}

\begin{document}

\maketitle

\begin{abstract}
We prove sharp uniform estimates for strong supersolutions of a large class of fully nonlinear degenerate elliptic complex equations. Our findings rely on ideas of Kuo and Trudinger who dealt with degenerate linear equations in the real setting. We also exploit the pluripotential theory for the complex Monge-Amp\`ere operator as well as suitably tailored theory of $L^p$-viscosity subsolutions.
\end{abstract}

\keywords{nonlinear degenerate elliptic equations, viscosity solutions, a priori estimates}

\vspace{0.2cm}
\MSC{35J70, 35D40, 32U15}


\section{Introduction}

Various {\it maximum principles} play pivotal role in the study of elliptic second order equations.
In the case of a uniformly elliptic equation\footnote{We shall use the sign convention so that $F(D^2u)=\Delta u$ in the case of the Laplace operator.}
$$F(D^2u)=f,$$
the basic version of the maximum principle says that $f>0$ implies that a (suitably regular) solution  $u$ does not achieve strict  local maximum.  
More quantitative versions  are also widely studied in the literature and we refer to \cite{PS07} for the details.
One of the cornerstones in this field is the Alexandrov-Bakelman-Pucci estimate which yields a uniform bound on $u\in C^2(\Om)\cap C^0(\clos{\Om})$ ($\Omega \subset \R^n$ is a bounded domain) solving
$$F(D^2u)\leq f,$$
with $F$ being a uniformly elliptic second order operator and $f \in C^0(\clos{\Om})$.
If $u\geq 0$ on $\partial\Om$ it reads

$$
\sup _{\Om} \, (-u)\leq C(n,\Om,F)\norm{f_+}_{L^n(\Om)},
$$
where $f_+=\max (f,0)$ denotes the positive part of $f$.
In this article, we shall denote by $C(n), C(n,\Om)$, etc. a positive number that may change from line to line but that depends only on the quantities indicated between the brackets.

Various generalizations to {\it strong} solutions\footnote{We recall that a strong solution is a function that belongs to $W^{2,r}_{\loc}(\Om)$ for some $r \geq 1$ and that satisfies the corresponding equation almost everywhere.} in $W^{2,n}(\Om)$ (see \cite{CW98}) or to viscosity solutions (see e.g. \cite[Theorem 3.2]{CC95}) are plentiful in the literature.

When the equation {\it fails} to be uniformly elliptic maximum principle still holds under reasonable minimal conditions, see \cite{RS64} and references therein.
When it comes to quantitative estimates for {\it degenerate elliptic equations} the available results are considerably more restrictive. In \cite{KT07}, the authors established the following estimate generalizing the Alexandrov-Bakelman-Pucci estimate:

\begin{theorem}\label{KuoTr}
Let $\Om \subset \R^n$ be a bounded domain.
Let
$$
Lu=\sum_{i,j=1}^n a^{ij}(x) u_{ij},
$$
be a second order linear operator with the coefficient matrix $A=(a^{ij})_{1 \leq i,j \leq n}$ being symmetric and positive definite.
We assume that $\rho_k^*(A)>0$ for some $k \in \ens{1,\ldots,n}$, where
$$\rho_k^*(A)=\inf\ens{\frac{1}{n}\Trace(A M) \quad \middle|\quad \sigma_k(\lambda(M))\geq\binom nk,  \quad \sigma_{\ell}(\lambda(M))>0, \quad \forall  \ell \in \ens{1,\cdots,k}},$$
where $\sigma_{\ell}$ is the $\ell$-th elementary symmetric polynomial and $\lambda(M) \in \R^n$ is the vector of eigenvalues of $M$.
Let $q \geq 1$ be such that
$$
\begin{cases}
q=k & \text{ if } k>n/2, \\
q>n/2 & \text{ if } k\leq n/2,
\end{cases}
$$
and $f$ be such that $f/\rho_k^*(A) \in L^q(\Omega)$.
Then, for any function $u\in W^{2,q}_{\loc}(\Om)\cap C^0(\clos{\Om})$ that satisfies
$$
\begin{cases}
\begin{aligned}
Lu & \leq f && \mbox{ in } \Omega, \\
u & \geq 0 && \mbox{ on } \partial\Omega,
\end{aligned}
\end{cases}
$$
we have
$$
\sup_{\Om} \, (-u) \leq C(n, \Om, q) \norm{\frac{f}{\rho_k^*(A)}}_{L^q(\Om)}.$$
\end{theorem}

In this note we shall investigate the complex analogues of Theorem \ref{KuoTr} and their generalizations to nonlinear complex equations. In such a setting it is well-known that direct application of real tools, such as Theorem \ref{KuoTr}, leads to non-optimal bounds in terms of the exponent $q$ (see \cite{Wan12,DD20}).
Instead, building on a fundamental theorem of Ko\l odziej \cite{Kol98} we are able to establish a fairly sharp  Alexandrov-Bakelman-Pucci type estimates.

\subsection{Assumptions on the class of nonlinear equations}\label{sect hyp}

Let us now detail the class of nonlinear complex elliptic equations that is considered all along this article.

From now on, $\Om \subset \C^n$ is a bounded domain.
The Euclidean norm of $z \in \C^n$ will be denoted by $\norm{z}$.
The open ball of center $z \in \C^n$ and radius $R>0$ will be denoted by $B_R(z)$.

Let $\Herm^n$ denote the set of all Hermitian $n\times n$ matrices and let us introduce the classical cone
$$
\coneb_n=\ens{A \in \Herm^n \quad \middle| \quad A>0}.
$$

In what follows we shall be interested in families $(\Gamma(z))_{z \in \Omega} \subset \Herm^n$ of open convex cones subject to the following condition:
\begin{equation}\label{cones range}
\coneb_n\subset\Gamma(z), \quad \forall z \in \Omega.
\end{equation}

Examples will shortly be presented in Section \ref{sect ex Hess} below.

For the rest of this article, $(\Gamma(z))_{z \in \Omega} \subset \Herm^n$ is now a fixed family of open convex cones satisfying \eqref{cones range}.
Let us then introduce the set
$$\Sigma=\ens{(z,A) \quad \middle| \quad z \in \Omega, \quad A \in \Gamma(z)}.$$
This set is clearly nonempty as it contains $(z,\Id)$ for any $z \in \Omega$ by \eqref{cones range}.

All along this work, we shall consider operators
$$G:\Sigma \longrightarrow \R,$$
(and $F=G^k$, $k>0$, for nonnegative $G$) subject to the following conditions:
\begin{enumerate}[(a)]
\item\label{reg F}
{\bf Regularity:} 
$\Sigma$ is measurable subset of $\Omega\times\Herm^n$ and $G$ is a measurable function on $\Sigma$.
Furthermore, for a.e. $z \in \Omega$, we have $A \mapsto G(z,A) \in C^1(\Gamma(z))$.

\item\label{F hom}
{\bf Homogeneity:} for a.e. $z \in \Omega$, the function $A \in \Gamma(z) \longmapsto G(z,A)$ is positively homogeneous of degree $1$ (that is $G(z,\alpha A)=\alpha G(z,A)$ for every $\alpha>0$ and $A \in \Gamma(z)$).

\item\label{concavity}
{\bf Concavity:} for a.e. $z \in \Omega$, the function $A \in \Gamma(z) \longmapsto G(z,A)$ is concave.

\item\label{comparison}
{\bf Comparison:}
for a.e. $z \in \Omega$, we have
\begin{equation}\label{comparison with det}
G(z,P) \geq (\det(P))^{\frac{1}{n}}, \quad \forall P \in \coneb_n.
\end{equation}
\end{enumerate}

Once again, some examples will be presented in Section \ref{sect ex Hess} below.

The key assumption is \ref{comparison}, which will allow a comparison with the Monge-Ampère equation.
Note that we do not require any regularity with respect to $z$.

\begin{remark}
The assumptions \ref{reg F}, \ref{F hom}, \ref{concavity} and \ref{comparison} are stable by finite convex combination.
More precisely, if $G_1,\ldots,G_\ell$ are operators satisfying these assumptions, then so is $G(z,A)=\sum_{i=1}^\ell \alpha_i(z) G_i(z,A)$ for any measurable functions $\alpha_i:\Omega \rightarrow \R$ with $\alpha_i \geq 0$ and $\sum_{i=1}^\ell \alpha_i=1$.
\end{remark}

In what follows, we will use the standard notation
$$G^{i\bar{j}}(z, A)=\frac{\partial G}{\partial a_{i\bar{j}}}(z,A), \quad 1 \leq i,j \leq n.$$

\begin{remark}\label{rem hyp comp with MA}
Inspired from an argument of \cite[p. 269]{CNS84}, we see that:
\begin{itemize}
\item
Thanks to the homogeneity assumption \ref{F hom}, the concavity assumption \ref{concavity} is equivalent to the following property, which will play the role of a substitute to \cite[Proposition 2.1]{KT07}:
\begin{enumerate}[(c')]
\item\label{lin concave}
For every $A \in\Gamma(z)$ and $B=(B_{i\bar{j}})_{1 \leq i,j \leq n} \in \Gamma(z)$, we have
\begin{equation}\label{basicineq}
\sum_{i,j=1}^n G^{i\bar{j}}(z,A) B_{i\bar{j}}
\geq G(z, B).
\end{equation}
\end{enumerate}
Indeed, introducing $DG(A)B=\sum_{i,j=1}^n G^{i\bar{j}}(z,A) B_{i\bar{j}}$, the concavity is equivalent to the inequality $DG(A) B \geq G(z,B)-G(z,A)+DG(A) A$ and we have the identity $DG(A) A=G(z,A)$ (obtained by differentiating $G(z,\alpha A)=\alpha G(z,A)$ with respect to $\alpha$ and taking $\alpha=1$).

\item
Thanks to the assumptions \ref{F hom} and \ref{concavity}, $G$ satisfies \ref{comparison} if, and only if,
\begin{enumerate}[(d')]
\item\label{not UE}
For every $A \in \Gamma(z)$ and $P \in \coneb_n$, we have
$$G(z,A+P) \geq G(z,A)+\left(\det(P)\right)^{\frac{1}{n}}.$$
\end{enumerate}
Indeed, the concavity inequality $G(z,A+P)-G(z,A) \geq DG(A+P) P$, \ref{lin concave} and \ref{comparison} imply \ref{not UE}; the converse is proved taking $A=P$ in \ref{not UE} and using the homogeneity.

In particular, note that our assumptions guarantee that $G$ is elliptic.
\end{itemize}
\end{remark}

\subsection{Main result and comments}

From now on, $D^2 u=\left(u_{i \bar{j}}\right)_{1 \leq i,j \leq n}$ denotes the complex Hessian of $u$, where we use the standard notations $u_j$ and $u_{\bar{j}}$ to denote, respectively, $\frac{\partial{u}}{\partial z_j}=\frac{1}{2}\left(\frac{\partial{u}}{\partial x_j}-i\frac{\partial{u}}{\partial y_j}\right)$ and $\frac{\partial{u}}{\partial \bar{z}_j}=\frac{1}{2}\left(\frac{\partial{u}}{\partial x_j}+i\frac{\partial{u}}{\partial y_j}\right)$.

The main result of the present paper is the following:

\begin{theorem}\label{main thm}
Let $(\Gamma(z))_{z \in \Omega} \subset \Herm^n$ be a family of open convex cones satisfying \eqref{cones range} and let $G$ be an operator satisfying \ref{reg F}, \ref{F hom}, \ref{concavity} and \ref{comparison}.
Let
\begin{equation}\label{cond rp}
r>n, \quad p>n.
\end{equation}
Let $g \in L^p(\Omega)$.
Then, for any function $u \in W^{2,r}_{\loc}(\Om)\cap C^0(\clos{\Om})$ with $(D^2 u)(z) \in \Gamma(z)$ for a.e. $z \in \Omega$ and that satisfies
$$
\begin{cases}
\begin{aligned}
G(z,\DD u) & \leq g && \mbox{ in } \Omega, \\
u & \geq 0 && \mbox{ on } \partial\Omega,
\end{aligned}
\end{cases}
$$
we have
\begin{equation}\label{max estim nonlin}
\sup_{\Om} \, (-u) \leq C(n, \diam \Om, r,p)
\norm{g_+}_{L^p(\Omega)},
\end{equation}
where $g_+=\max (g,0)$ denotes the positive part of $g$.
\end{theorem}

Note that this result can be in particular applied to linear equations. Then (for suitable $G$) it can be seen as a complex counterpart of Theorem \ref{KuoTr}.
 
For nonlinear equations with operators which are positively homogeneous of degree different from $1$, we have the following immediate consequence:

\begin{corollary}\label{main thm k}
Let $(\Gamma(z))_{z \in \Omega} \subset \Herm^n$ be a family of open convex cones satisfying \eqref{cones range}.
Let $F:\Sigma \rightarrow \R$ be a nonnegative function such that, for some $k,\delta>0$, $G=(F^{1/k})/\delta$ satisfies \ref{reg F}, \ref{F hom}, \ref{concavity} and \ref{comparison}.
Let
\begin{equation}\label{cond rpk}
r>n, \quad p\geq 1, \quad p>\frac{n}{k}.
\end{equation}
Let $f \in L^p(\Omega)$ with $f \geq 0$ in $\Omega$.
Then, for any function $u \in W^{2,r}_{\loc}(\Om)\cap C^0(\clos{\Om})$ with $(D^2 u)(z) \in \Gamma(z)$ for a.e. $z \in \Omega$ and that satisfies
$$
\begin{cases}
\begin{aligned}
F(z,\DD u) & \leq f && \mbox{ in } \Omega, \\
u & \geq 0 && \mbox{ on } \partial\Omega,
\end{aligned}
\end{cases}
$$
we have
$$
\sup_{\Om} \, (-u) \leq C(n, \diam \Om, r, p, k, \delta)
\norm{f}_{L^p(\Omega)}^{\frac{1}{k}}.
$$
\end{corollary}

This establishes Ko\l odziej type uniform bounds (\cite{Kol98}) for a large class of nonlinear equations.
Note however that Corollary \ref{main thm k} is a uniform estimate for strong solutions, whereas in the case of the Monge-Ampère equation, the result from \cite{Kol98} is valid for more general solutions.
We also recall that for the Monge-Ampère equation the condition \eqref{cond rpk} for $p$, which becomes $p>1$, is optimal.

\begin{remark}
All our arguments apply verbatim also to $L^p$-viscosity solutions once a comparison principle is established\footnote{The comparison principle states that a subsolution is majorized by a supersolution once this holds on the boundary.}.
It is worth emphasizing that such a comparison principle is lacking even for general uniformly elliptic equations, see e.g. \cite{CCKS96}.
We refer to \cite{JS05} for the up-to-date partial results on the comparison principle for uniformly elliptic equations and to \cite{ABT07} for analogous discussion in the special case of the real Monge-Amp\`ere equation.
\end{remark}

\begin{remark}
We wish to point out that the methods from \cite{KT07} are also applicable in the complex setting but yield estimates dependent on $\norm{f}_{L^{2p}(\Om)}$.
The improvement in the exponent should be compared with the result in \cite{Esc93} who improved thel regularity assumptions in the setting of uniformly elliptic equations.
\end{remark}

\begin{remark}
In case we only know that $(D^2 u)(z) \in \clos{\Gamma(z)}$ for a.e. $z \in \Omega$, Theorem \ref{main thm} and its proof remain unchanged provided that there exists a set $N \subset \R$ such that, for a.e. $z \in \Omega$, $G(z,\cdot)$ is $C^1$ in a neighborhood of $\ens{A \in \clos{\Gamma(z)}, \quad G(z,A) \not\in N}$ and $G(z,(D^2 u)(z)) \not\in N$.
\end{remark}

Our proof is based on $L^p$-viscosity techniques suitably coupled with basic results from pluripotential theory.
Roughly speaking we produce a pluripotential subsolution to our equation and then show that it is also an $L^p$-viscosity barrier.
This coupled with a maximum principle for $L^p$-viscosity subsolutions yields the claim.
Our findings in fact show that uniform estimates for $L^p$-viscosity solutions to a large class of Hessian type equations can be deduced through pluripotential theoretic tools even if a pluripotential theory cannot be developed for
a particular equation (see \cite{Din20-pre} for a discussion of such a phenomenon).

The rest of this paper is organized as follows.
In Section \ref{sect ex Hess} we give some examples of Hessian equations that are covered by Theorem \ref{main thm} or Corollary \ref{main thm k}.
In Section \ref{sect visc basic} we introduce the notion of $W^{2,r}/L^p$-viscosity subsolutions for general elliptic equations.
In particular, Section \ref{sect examples} is devoted to examples explaining the differences from standard $L^p$-viscosity theory in the absence of uniform ellipticity.
In Section \ref{sect pluri} we show that pluripotential subsolutions to the complex Monge-Amp\`ere equation with $L^p$ right-hand side are also $W^{2,r}/L^p$-viscosity subsolutions when $r>n$.
In Section \ref{sect max PP} we prove a basic $W^{2,r}/L^p$-viscosity maximum principle.
The final Section \ref{sect proof} is devoted to the proof of our main result Theorem \ref{main thm}.

\subsection{Some examples for Hessian equations}\label{sect ex Hess}

In this work, our operators $F$ are not necessarily {\it Hessian} but all our examples below will be.
For this reason, let us first recall the notion of {\it Hessian type operators} on $\Om$.

In this section, $\beta$ will be a fixed smooth positive Hermitian $(1,1)$-form.
We recall that, given a smooth real $(1,1)$-form $\alpha$ on $\Om$ (not necessarily positive) the {\it eigenvalues} of $\alpha$ with respect to $\beta$ at a point $z$ are the solutions $\lambda$ to the equation
$$(\alpha(z)-\lambda\beta(z))^n=0.$$
These eigenvalues are real and they will be denoted by $\lambda_1(z,\alpha) \leq \ldots \leq \lambda_n(z,\alpha)$ and arranged in
$$\lambda(z,\alpha)=\left(\lambda_1(z,\alpha), \cdots, \lambda_n(z,\alpha)\right).$$

We recall that the eigenvalues are continuous functions of $z$ which are furthermore smooth off the branching locus.
Equivalently, writing $\beta=i\sum_{j,k=1}^n B_{j\bar{k}}dz_j\wedge d\bar{z}_k$ and $\alpha=i\sum_{j,k=1}^n A_{j\bar{k}}dz_j\wedge d\bar{z}_k$, where $A=(A_{j\bar{k}})_{1 \leq j,k \leq n}$ and $B=(B_{j\bar{k}})_{1 \leq j,k \leq n}$ are Hermitian matrices ($d$ is the exterior derivative), the eigenvalues are solutions to $\det(A(z)-\lambda B(z))=0$.

\begin{remark}
The choice $\beta=i\sum_{j=1}^n dz_j\wedge d\bar{z}_j$ (equivalently $B(z)=\Id$) yields the standard eigenvalues of a real $(1,1)$-form.
In applications however, especially when studying locally equations on complex manifolds, it is more natural to work with a $z$-dependent background form (see e.g. \cite[p. 230]{DK14}).
\end{remark}

If $M \in\Herm^n$, we simply denote by $\la_j(z,M)$ the eigenvalues of the corresponding form $\alpha=i\sum_{j,k=1}^n M_{j\bar{k}}dz_j\wedge d\bar{z}_k$ with respect to $\beta$ at the point $z$.
When $B$ does not depend on $z$, we shall simply write $\la_j(M)$ and $\lambda(M)$.

Let us now give a precise statement of what we call an Hessian operator:

\begin{definition}\label{Hessianoperator}
Assume that the family $(\Gamma(z))_{z \in \Omega} \subset \Herm^n$ of open convex cones satisfies \eqref{cones range} and
\begin{equation}\label{cone prop 3}
\forall z \in \Om, \, \forall A, \tilde{A} \in \Herm^n,
\quad
\left(A\in \Gamma(z) \text{ and } \lambda(z,\tilde{A})=\lambda(z,A)\right)
\Longrightarrow \tilde{A}\in \Gamma(z).
\end{equation}
Then, a function $F:\Sigma \longrightarrow \R$ is said to be a Hessian operator if there is exists a function $\hat{F}:\hat{\Sigma} \longrightarrow \R$, where $\hat{\Sigma}=\ens{(z,\lambda(z,A)) \quad \middle| \quad (z,A) \in \Sigma}$, so that, for every $(z,A) \in \Sigma$, we have
$$F(z,A)=\hat{F}(z,\la(z,A)).$$
When $\hat{F}$ does not depend on $z$, we shall simply write $\hat{F}(\lambda(z,A))$.
\end{definition}

In the case $B(z)=\Id$, the property \eqref{cone prop 3} can be equivalently rephrased as the $\mathcal{O}(n)$-invariance ($\mathcal{O}(n)$ denotes the orthogonal group): we have $A\in\Gamma(z)$ if, and only if, 
\begin{equation}\label{orth inv}
O^*AO \in \Gamma(z), \quad \forall O\in\mathcal{O}(n),
\end{equation}
where $O^*$ denotes the Hermitian transposed matrix of $O$.
This property (\ref{orth inv}) appears for instance in \cite[p. 778]{HL18}, where it is called ``ST-Invariance'' (it stands for Spherical Transitivity).

Before finally presenting some examples, we recall the definition of the basic cones associated to Hessian equations.
For $1 \leq m \leq n$, the cones $\Gamma_m \subset \R^n$ are defined as follows
$$
\Gamma_m=\ens{(\lambda_1,\ldots,\lambda_n) \in \R^n \quad \middle| \quad \sigma_q(\lambda_1,\ldots,\lambda_n)>0, \quad \forall  q \in \ens{1,\cdots,m}},
$$
where
$$
\sigma_q(\lambda_1,\ldots,\lambda_n)=\sum_{1\leq i_1<\cdots<i_q\leq n}\la_{i_1}\cdots\la_{i_q}.
$$

The cones $\lambda^{-1}(\Gamma_k)$ clearly satisfy the $\mathcal{O}(n)$-invariance \eqref{orth inv} and it can be shown that they are convex (see e.g. \cite[Section 2]{Blo05}).

\begin{example}\label{ex cones}
Examples of equations covered by our framework of Section \ref{sect hyp} and that are Hessian equations include:
\begin{enumerate}[1)]
\item
The complex Monge-Ampère equation:
$$
\Gamma(z)=\lambda^{-1}(\Gamma_n),
\quad
\hat{F}(\lambda_1,\ldots,\lambda_n)=\prod_{i=1}^n \lambda_i.
$$
Here, the degree of homogeneity of $F$ is $k=n$, the concavity is well known and the comparison \ref{comparison} is trivial.

\item
The complex $m$-Hessian operator: for $1 \leq m \leq n$,
$$
\Gamma(z)=\lambda^{-1}(\Gamma_m),
\quad
\hat{F}=\sigma_m.
$$
Here, the degree is $k=m$, the concavity follows from G{\aa}rding's inequality, and the comparison \ref{comparison} follows from Maclaurin's inequality (see e.g. \cite[Section 2]{Blo05}).

\item
The complex $m$-Monge-Ampère operator (c.f. \cite{HL18,Din20-pre,Din20-pre-2}): for $1 \leq m \leq n$,
\begin{gather*}
\Gamma(z)=\lambda^{-1}\left(\ens{(\lambda_1,\ldots,\lambda_n) \in \R^n \quad \middle| \quad \la_{i_1}+\cdots+\la_{i_m}>0, \quad 
\forall 1\leq i_1<\cdots<i_m\leq n}\right),
\\
\hat{F}(\lambda_1,\ldots,\lambda_n)=\prod_{1\leq i_1<\cdots<i_m\leq n}(\la_{i_1}+\cdots+\la_{i_m}).
\end{gather*}
Here, the degree is $k=\binom nm$, the operator is concave and the comparison \ref{comparison} holds (see e.g. \cite[Section 1.6]{Din20-pre-2} and \cite{AO20}).

\item
$B(z)=\Id$ and, for $a\in [0,1]$,
\begin{gather*}
\Gamma(z)=\lambda^{-1}(\Gamma_{2-a}),
\quad
\Gamma_{2-a}=\ens{(\lambda_1,\lambda_2) \in \R^2 \quad \middle| \quad \la_1+a\la_2>0, \quad \la_2+a\la_1>0},
\\
\hat{F}(\lambda_1,\lambda_2)=(1-a)^2\lambda_1 \lambda_2+a(\lambda_1+\lambda_2)^2,
\end{gather*}
(the cones $\Gamma_{2-a}$ interpolate between $\Gamma_1$ and $\Gamma_2$).
Here, the degree is $k=2$, the concavity and the comparison \ref{comparison} are easily checked.

\item
More generally, for any polynomial $P$ hyperbolic with respect to $v\in\R^n$ (cf. \cite{Gar59, CNS85, HL18}) of degree $k$, the operator $\hat{F}(\lambda_1,\ldots,\lambda_n)=P(\lambda_1,\ldots,\lambda_n)$ defined on the component of $P \neq 0$ in $\R^n$ containing $v$, satisfies \ref{concavity}.
Whether it satisfies the comparison \ref{comparison} or not may depend on $P$.

\item
An example of an operator meeting all the conditions above {\bf except} \ref{comparison} is the Hessian quotient operator given by
$$
\Gamma(z)=\lambda^{-1}(\Gamma_m),
\quad
\hat{F}(\lambda_1,\ldots,\lambda_n)=\frac{\sigma_m(\lambda_1,\ldots,\lambda_n)}{\sigma_\ell(\lambda_1,\ldots,\lambda_n)},
$$
for $1\leq \ell<m\leq n$.
Here, the degree is $k=m-\ell$ and the operator is concave.
\end{enumerate}
\end{example}

\section{$W^{2,r}/L^p$-viscosity subsolutions}\label{sect visc basic}

As our analysis will be based on the pluripotential theory for the complex Monge-Amp\`ere equation, we need to recall that plurisubharmonic solutions to this equation, even for regular right-hand side data, need not possess sufficient Sobolev regularity (see \cite{BT82,Blo99,Kol05,DD20} and Section \ref{sect examples} below).
Hence they are {\it not} strong solutions in general and as such cannot be tested directly for other types of operators.
On the other hand, as discussed in \cite{Din20-pre}, a general Hessian operator, even one satisfying the hypotheses above, may fail to have properties necessary to develop its own pluripotential theory.
Hence in order to accommodate pluripotential (sub)solutions  it is necessary to develop another theory of weak solutions.
As it turns out a good choice is the $L^p$-viscosity theory which we shall briefly sketch below.
We refer for instance to \cite{CCKS96} for more background.

In all this section, we only need to assume that $A \mapsto F(z,A)$ is elliptic in $\Gamma(z)$ for a.e. $z \in \Omega$.

\subsection{Definition and remarks}

Below we introduce the notion of a $W^{2,r}/L^p$-viscosity subsolution associated to an operator $F$.
It is inspired from \cite[Definition 2.1]{CCKS96} but also contains notable differences, see Remark \ref{rem def visc} below.

First of all, when dealing with constrained elliptic equations (i.e. $\Gamma(z) \neq \Herm^n$) it is a standard and useful procedure in viscosity theory (see e.g. \cite{CIL92,ABT07}) to extend the operator $F$ by $-\infty$ outside the cone i.e.
\begin{equation}\label{offcone}
F(z,A)=-\infty, \quad \forall z \in \Omega, \, \forall A \in \Herm^n \backslash \Gamma(z).
\end{equation}

\begin{definition}\label{def visc sol}
Let $F:\Sigma \longrightarrow \R$ satisfy \eqref{offcone}, $r,p \geq 1$ and $f \in L^p(\Omega)$.
An upper semicontinuous function $u:\Om \longrightarrow \R$ is said to be a $W^{2,r}/L^p$-viscosity solution of
$$F(z, \DD u)\geq f \quad \mbox{ in } \Omega,$$
if, for every lower semicontinuous function $\varphi:\Om \longrightarrow \R$ with $\varphi \in W^{2,r}_{\loc}(\Om)$, for every $\varepsilon>0$ and nonempty open subset $U\subset\Om$, if
$$F(z,(D^2\varphi)(z))\leq f(z)-\varepsilon \quad \mbox{ a.e. } z \in U,$$
then $u-\varphi$ cannot have a {\bf strict} local maximum in $U$.
\end{definition}

We will also say that $u$ is a $W^{2,r}/L^p$-viscosity subsolution.

\begin{remark}\label{rem def visc}
Let us comment this new definition:
\begin{itemize}
\item
The functions $\varphi$ in Definition \ref{def visc sol} are called test functions.
If the class of $W^{2,r}_{\loc}$ testing functions above is exchanged by the class of $C^2$ tests we get the notion of $C$-viscosity (sub)solutions.
These are much better studied and we refer to \cite{CIL92,CC95} for the general theory of $C$-viscosity for uniformly elliptic equations.
For complex equations, $C$-viscosity has been considered first in \cite{EGZ11} and then in \cite{Wan12,Zer13} for the complex Monge-Ampère equation and in \cite{DDT19} for degenerate Hessian equations.

\item
The main difference with the definition introduced in \cite[Definition 2.1]{CCKS96} is that our operators are not uniformly elliptic and this is why we require in addition the condition that the maximum is strict.
We refer to Section \ref{sect examples} for some instructive examples.

\item
In Theorem \ref{main thm} we assume that $r>n$, for which we have the
Sobolev embedding $W^{2,r}_{\loc}(\Om) \subset C^0(\Om)$.
However, there is no need to require such a condition in Definition \ref{def visc sol} since every quantity makes sense as it is introduced (compare with \cite[Definition 2.1]{CCKS96}).

\item If $F$ is a linear operator of  type $F(z,D^2u)=\sum_{i,j=1}^n a^{i\bar{j}}u_{i\bar{j}}$ with $(a^{i\bar{j}})_{1 \leq i,j \leq n} \geq 0$, contrary to the uniformly elliptic case (see \cite{CCKS96}) we do {\bf not} assume that the matrix entries are essentially bounded, and hence $F$ is in general {\bf not} expected to send $W^{2,p}$ functions $u$ to $L^p$.
\end{itemize}
\end{remark}

Finally, the following observation will be used in the proof of the main theorem.

\begin{remark}\label{sum visc sol}
If $A \mapsto G(z,A)$ is linear and if $u_1$ is a $W^{2,r}/L^p$-viscosity solution to $G(z,D^2 u_1) \geq g_1$ in $\Omega$ and $u_2$ is a $W^{2,r}/L^p$-viscosity solution to $G(z,D^2 u_2) \geq g_2$ in $\Omega$  with $u_2 \in W_{\loc}^{2,r}(\Omega)$, then $u=u_1+u_2$ is a $W^{2,r}/L^p$-viscosity solution to $G(z,D^2 u) \geq g_1+g_2$ in $\Omega$.
This follows from the definition as we can subtract $u_2$ from any testing function $\varphi\in W^{2,r}_{\loc}(\Om)$ and the resulting function is testing for $u_1$.
\end{remark}

\subsection{The strict maximum condition}\label{sect examples}

In this section, we explain why we have required that the maximum is strict in Definition \ref{def visc sol}.
We discuss this for the complex Monge-Ampère equation and its linearization.

Recall once again that equations of Monge-Amp\`ere type locally admit {\it singular} solutions no matter how smooth the right-hand side is.
The following example from \cite{Blo99} is modeled on real Pogorelov type singular convex functions:

\begin{example}\label{ex Blocki}
Let $n \geq 2$.
Let $u$ be the function defined for every $z=(z_1,z')\in\C^n$ by
\begin{equation}\label{def u ex}
u(z)= \norm{z'}^{2\left(1-\frac{1}{n}\right)}\left(1+|z_1|^2\right).
\end{equation}
It is smooth off $\ens{z'=0}$ and $u \in W^{2,r}_{\loc}(\C^n)$ for any $1 \leq r<n(n-1)$ with $(D^2 u)(z) \in \coneb_n$ for every $z \in \C^n \backslash \ens{z'=0}$.
A computation shows that it is a strong solution to
\begin{equation}\label{sol u ex}
\det(D^2 u)=f \quad \text{ in } \C^n, \quad f(z)=\left(1-\frac{1}{n}\right)^n(1+|z_1|^2)^{n-2}.
\end{equation}
\end{example}

It is also easily seen that $u$ is a $C$-viscosity solution, the only problem is at $\ens{z'=0}$ and there are clearly no $C^2$-smooth differential tests from above, while $C^2$ differential tests from below have to vanish along $\ens{z'=0}$ and thus have vanishing Monge-Amp\`ere operator \cite{Zer13}.

When it comes to $W^{2,r}/L^p$-viscosity things are substantially subtler.
We have built the following example:

\begin{example}\label{ex strict needed}
Let $n\geq 3$ and $R>0$ be fixed.
Let $u$ and $f$ be the functions of Example \ref{ex Blocki}.
Consider the function $\varphi$ defined for every $z=(z_1,z')\in\C^n$ by
\begin{equation}\label{def varphi ex}
\varphi(z)=\norm{z'}^{2\left(1-\frac{1}{n}\right)}(1+R^2-\norm{z'}^2).
\end{equation}
We have $\varphi \in W^{2,r}_{\loc}(B_R(0))$ for every $n<r<n(n-1)$ and $(D^2 \varphi)(z) \in \coneb_n$ for every $z \in B_R(0) \backslash \ens{z'=0}$.
As $\varphi$ is a function of $n-1$ variables, it is a strong solution to the equation
$$\det(D^2\varphi)=0 \quad \text{ in } B_R(0).$$
Besides, we can always find $\epsilon>0$ small enough so that $0<f-\varepsilon$.
On the other hand, we clearly have
$$u-\varphi \leq 0=(u-\varphi)_{| \ens{z'=0}}, \quad \text{ in } B_R(0),$$
so that $u-\varphi$ has no strict maximum in $B_R(0)$.

In conclusion, if we do not require the strict maximum condition in Definition \ref{def visc sol}, then the function $u$ is a strong solution to $\det(D^2 u)=f$ but it would not be a $W^{2,r}/L^p$-viscosity solution to $\det(D^2 u) \geq f$, whatever $p \geq 1$ is (even if $r>n$).
Note that this has nothing to do with the regularity of the right-hand side since $f \in C^{\infty}(\C^n)$.
\end{example}

The previous example shows that imposing {\it strict} local maxima in the definition of $W^{2,r}/L^p$-viscosity solutions is essential for Hessian type equations.
It turns out that the same example works for the {\it linearized equation} provided $r$ is taken sufficiently small.

\begin{example}\label{linearizedexample}
Let $u,f,\varphi$ be the same functions as in Example \ref{ex strict needed} but consider this time the following linear operator:
$$L_u h=\sum_{i,j=1}^n a^{i\bar{j}}(z) h_{i\bar{j}},
\quad a^{i\bar{j}}(z)=\frac{\partial \det}{\partial a_{i\bar{j}}}((D^2 u)(z)).$$
By homogeneity, we immediately have $L_u u=n \det(D^2 u)=nf$.
On the other hand, for sufficiently small $R,\varepsilon>0$, we will show that
\begin{equation}\label{ineq Lu varphi}
L_u \varphi \leq n f-\varepsilon \quad \text{ in } B_R(0).
\end{equation}
Therefore, even for linear equations, the strict maximum condition that we introduced in  Definition \ref{def visc sol} is needed if we at least want that strong solutions are $W^{2,r}/L^p$-viscosity solutions.

We provide the details of the inequality \eqref{ineq Lu varphi} for $n=3$ only, the general case is analogous.
An explicit computation of the $a^{i\bar{j}}$ and the fact that $\varphi$ does not depend on $z_1$ yield
\begin{multline*}
 \sum_{i,j=1}^3 a^{i\bar{j}}(z)\varphi_{i\bar{j}}
=(u_{1\bar{1}}u_{3\bar{3}}-|u_{1\bar{3}}|^2)\varphi_{2\bar{2}}
+ (u_{1\bar{1}}u_{2\bar{2}}-|u_{1\bar{2}}|^2)\varphi_{3\bar{3}}
\\
+(u_{3\bar{1}}u_{1\bar{2}}-u_{1\bar{1}}u_{3\bar{2}})\varphi_{2\bar{3}}
+(u_{2\bar{1}}u_{1\bar{3}}-u_{1\bar{1}}u_{2\bar{3}})\varphi_{3\bar{2}}.
\end{multline*}
From the expression \eqref{def u ex} of $u$ we have, for $z' \neq 0$,
\begin{gather*}
u_{1\bar{1}}u_{3\bar{3}}-|u_{1\bar{3}}|^2
=\frac{2}{3}\left(1+\abs{z_1}^2\right)\norm{z'}^{\frac{2}{3}}
-\left(\frac{2}{9}+\frac{2}{3}\abs{z_1}^2\right)|z_3|^2\norm{z'}^{-\frac{4}{3}},
\\
u_{1\bar{1}}u_{2\bar{2}}-|u_{1\bar{2}}|^2
=\frac{2}{3}\left(1+\abs{z_1}^2\right)\norm{z'}^{\frac{2}{3}}
-\left(\frac{2}{9}+\frac{2}{3}\abs{z_1}^2\right)|z_2|^2\norm{z'}^{-\frac{4}{3}},
\\
u_{3\bar{1}}u_{1\bar{2}}-u_{1\bar{1}}u_{3\bar{2}}
=\left(\frac{2}{9}+\frac{2}{3} \abs{z_1}^2\right) \bar{z}_3 z_2 \norm{z'}^{-\frac{4}{3}},
\\
u_{2\bar{1}}u_{1\bar{3}}-u_{1\bar{1}}u_{2\bar{3}}
=\left(\frac{2}{9}+\frac{2}{3} \abs{z_1}^2\right) \bar{z}_2 z_3 \norm{z'}^{-\frac{4}{3}}.
\end{gather*}
From the expression \eqref{def varphi ex} of $\varphi$ we have, for $z' \neq 0$,
\begin{gather*}
\varphi_{2\bar{2}}=
-\frac{5}{3}\norm{z'}^{\frac{4}{3}}
+\left(-\frac{10}{9}|z_2|^2 +\frac{2}{3}(1+R^2) \right)\norm{z'}^{-\frac{2}{3}}
-\frac{2}{9}(1+R^2) |z_2|^2\norm{z'}^{-\frac{8}{3}},
\\
\varphi_{3\bar{3}}=
-\frac{5}{3}\norm{z'}^{\frac{4}{3}}
+\left(-\frac{10}{9}|z_3|^2 +\frac{2}{3}(1+R^2) \right)\norm{z'}^{-\frac{2}{3}}
-\frac{2}{9}(1+R^2) |z_3|^2\norm{z'}^{-\frac{8}{3}},
\\
\varphi_{2\bar{3}}=
-\frac{10}{9}\bar{z}_2z_3\norm{z'}^{-\frac{2}{3}}
-\frac{2}{9}(1+R^2)\bar{z}_2 z_3\norm{z'}^{-\frac{8}{3}},
\\
\varphi_{3\bar{2}}=
-\frac{10}{9}\bar{z}_3z_2\norm{z'}^{-\frac{2}{3}}
-\frac{2}{9}(1+R^2)\bar{z}_3 z_2\norm{z'}^{-\frac{8}{3}}.
\end{gather*}
Therefore, for $z' \neq 0$,
\begin{align*}
L_u \varphi &=
-\left(\frac{70}{27}+\frac{50}{27}\abs{z_1}^2\right)\norm{z'}^2
+\frac{16}{27}(1+R^2)+\frac{8}{27}(1+R^2)|z_1|^2 \\
& \leq \frac{16}{27}(1+R^2)+\frac{8}{27}(1+R^2)|z_1|^2.
\end{align*}
As $3f=\frac{24}{27}+\frac{24}{27}|z_1|^2$ (recall \eqref{sol u ex}), we see that $L_u \varphi<3f$ for any $R<\frac 1{\sqrt{2}}$, so that \eqref{ineq Lu varphi} holds provided $\varepsilon$ is taken small enough (since $L_u \varphi$ and $f$ are continuous on the closure of $B_R(0)$).
\end{example}

\section{$W^{2,r}/L^p$-viscosity subsolutions and pluripotential theory}\label{sect pluri}

All along this section, we denote the open unit ball by $B_1=B_1(0)$.

Let $F_{\MA}(A)=\det A$ for $A \in \coneb_n$ and $F_{\MA}(A)=-\infty$ otherwise.
The goal of this section is to show the following result:

\begin{theorem}\label{thm rho}
Let
\begin{equation}\label{cond rq}
r>n, \quad q>1,
\end{equation}
and let $g \in L^q(B_1)$ with $g \geq 0$.
Then, there exists a $W^{2,r}/L^q$-viscosity solution $\rho \in C^0(\clos{B_1})$ of
$$F_{\MA}(\DD \rho) \geq g \quad \mbox{ in } B_1,$$
with $\rho=0$ on $\partial B_1$ and which satisfies
$$
\sup_{B_1} \, (-\rho)\leq C(n,q) \norm{g}_{L^q(B_1)}^{\frac{1}{n}}.
$$
\end{theorem}

\subsection{Background on pluripotential theory}

First of all, let us discuss some  pluripotential tools.
One of the main problems which pluripotential theory  handles is the solvability complex Monge-Amp\`ere  equation. 
For every function $\rho \in \PSH(B_1)\cap C^2(B_1)$, the complex Monge-Amp\`ere operator is given as follows:
$$
(\ddc \rho)^n=
4^n n! \det\left(D^2 \rho\right) \dV.
$$
Here $\dV$ denotes the Lebesgue measure in  $\C^n$, $d=\partial +\bar{\partial}$, $d^c=i(\bar{\partial}-\partial)$, $\PSH(B_1)$ is the set of plurisubharmonic functions in $B_1$ i.e. the set of upper semicontinuous, locally integrable functions $\rho$ such that $\ddc \rho\geq 0$.
The operator $(\ddc \rho)^n$ is well defined on bounded plurisubharmonic functions, as follows from the work \cite{BT76} (see also \cite{BT82}).

Below we consider the Dirichlet problem associated to the Monge–Amp\`ere operator in the ball $\B$.
It reads:
\begin{equation}\label{MAdir}
\begin{cases}
(\ddc \rho)^n=g \dV & \text{in } \B,\\
\rho=\psi & \text{on } \partial\B,
\end{cases} 
\end{equation}
where $g \in L^q(\B)$ ($q>1$) and $\psi\in C^0(\partial\B)$.
We recall now the following fundamental result  from \cite{Kol98}:

\begin{theorem}\label{Kolodziej}
Let $g \in  L^q(\B)$ ($q>1$) with $g \geq 0$.
There exists a unique solution $\rho \in \PSH(\B) \cap C^0(\clos{\B})$ of \eqref{MAdir} with $\psi=0$, and it satisfies
$$\sup_{\B} \, (-\rho)\leq C(n,q) \norm{g}^{\frac{1}{n}}_{L^q(\B)}.$$
\end{theorem}

Consequently, we see that Theorem \ref{thm rho} will be a straightforward consequence of this result if we manage to establish the following:

\begin{proposition}\label{plpvslpvisc}
Assume \eqref{cond rq} and let $g \in L^q(B_1)$ with $g \geq 0$.
If $\rho \in \PSH(\B) \cap C^0(\clos{\B})$ is a solution to $(\ddc \rho)^n\geq 4^n n!g \dV$ in $B_1$ in the sense of currents, then $\rho$ is a $W^{2,r}/L^q$-viscosity solution to $F_{\MA}(\rho) \geq g$ in $B_1$. 
\end{proposition}

The proof of this proposition is the purpose of the next section.

\subsection{Comparison theorem for non $\PSH$ functions}

In this part, $B$ is any fixed open ball.
In order to prove Propostion \ref{plpvslpvisc}, we will use the following comparison theorem between $\PSH$ and $W^{2,r}$ functions, which is adapted from \cite[Theorem 5.1]{RT77} to our complex setting:

\begin{theorem}\label{thm RT}
Let $v \in W^{2,r}(B)$ with $r>n$.
Let $u \in \PSH(B)\cap C^0(\clos{B})$ satisfy
$$(\ddc u)^n \geq \left((\ddc v)^n\right)^{+} \quad \text{ in } B,$$
in the sense of currents, where $\left((\ddc v)^n\right)^{+}$ is by definition equal to $(\ddc v)^n$ if the current $\ddc v \geq 0$ (i.e. $D^2 v \in \clos{\coneb_n}$) and is equal to zero otherwise.
Then, we have
\begin{equation}\label{max on the b}
\max_{z \in B} \, (u(z)-v(z))
=\max_{z \in \partial B} \, (u(z)-v(z)).
\end{equation}
\end{theorem}

We would to emphasize that an essential difference in the proof of this result below compared with the one of \cite[Theorem 5.1]{RT77} is that the complex Monge-Amp\`ere operator, contrary to the real one, is not continuous with respect to the weak convergence (see e.g. \cite[Section 3.8]{Kli91}).
We will circumvent this by exploiting the stability properties of the Monge-Amp\`ere operator.
This in particular is one of the reasons for the assumption $r>n$ in this section.

Before proving Theorem \ref{thm RT}, let us show how it leads to Proposition \ref{plpvslpvisc}:

\begin{proof}[Proof of Proposition \ref{plpvslpvisc}]
Let $\rho \in \PSH(\B) \cap C^0(\clos{\B})$ be a solution to $(\ddc \rho)^n\geq 4^n n!g \dV$ in $B_1$ in the sense of currents.
Let $\varphi \in W^{2,r}_{\loc}(B_1)$ and a nonempty open subset $U\subset B_1$ ($\varphi$ is continuous since $r>n$).
Assume that $u-\varphi$ has a strict maximum, say at $z_0$ and in some small ball $\clos{B_R(z_0)} \subset B_1$.
Assume by contradiction that
\begin{equation}\label{equ FMA}
F_{\MA}((D^2\varphi)(z))\leq g(z) \quad \mbox{ a.e. } z \in B_R(z_0).
\end{equation}
If we show that this implies $(\ddc \rho)^n \geq \left((\ddc \varphi)^n\right)^{+}$ in $B_R(z_0)$, then Theorem \ref{thm RT} gives a contradiction.
We have to distinguish two cases.
If the current $\ddc \varphi$ is not nonnegative, then $\left((\ddc \varphi)^n\right)^{+}=0$ by definition and the desired inequality holds since $(\ddc \rho)^n\geq 4^n n!g \dV$ and $g \geq 0$.
Assume then that the current $\ddc \varphi$ is nonnegative, that is $D^2 \varphi \in \clos{\coneb_n}$.
We have $\left((\ddc \varphi)^n\right)^{+}=4^n n! \det\left(D^2 \varphi \right) \dV$ by definition and the desired inequality will be proved if we show that $g \geq \det\left(D^2 \varphi \right)$ in $B_R(z_0)$.
On the measurable subset $\ens{z \in B_R(z_0) \quad \middle| \quad (D^2 \varphi)(z) \in \coneb_n}$ we have $F_{\MA}(D^2 \varphi)=\det(D^2 \varphi)$, so that the previous inequality holds thanks to \eqref{equ FMA}.
On the remaining set we have $D^2 \varphi \in \partial \coneb_n=\ens{0}$, so that previous inequality holds as well thanks to $g \geq 0$.

\end{proof}

For the proof of Theorem \ref{thm RT}, we need to recall two results.
The first one is the following comparison principle that follows from \cite[Theorem 4.1]{BT82} and which is by now a basic tool in pluripotential theory.

\begin{theorem}\label{compa BT}
Let $ u,v\in \PSH(B)\cap C^0(\clos{B})$ satisfy
$$(\ddc u)^n \geq (\ddc v )^n \quad \text{ in } B,$$
in the sense of currents.
Then, the same equality as in \eqref{max on the b} holds.
\end{theorem}

The second result that will be needed is \cite[Théorème 1]{Sib77} coupled with \cite[Theorem 3.9]{Blo96}:

\begin{theorem}\label{sibony}
Let $v\in C^2(B)\cap C^0(\clos{B})$.
Let $u\in \PSH(B)\cap C^0(\clos{B})$ satisfy
$$(\ddc u)^n \geq (\ddc v )^n \quad \text{ in } B,$$
in the sense of currents.
Then, the same equality as in \eqref{max on the b} holds.
\end{theorem}

Note that Sibony's theorem states (in modern terminology) that solutions of $(\ddc u)^n\geq f \dV$ (for continuous $f$) in the sense of currents are $C$-viscosity subsolutions (see the proof of Proposition \ref{plpvslpvisc}).

We have now all the ingredients to prove the desired comparison theorem.

\begin{proof}[Proof of Theorem \ref{thm RT}]
We adapt the proof of \cite[Theorem 5.1]{RT77}.
First of all, we note that there exists a sequence $(v_j)_j \subset C^{\infty}(\C^n)$ such that
\begin{gather}
v_j \to v \text{ in } W^{2,r}(B), \label{CV vj v} \\
v_j \rightarrow v \text{ in } C^0(\clos{B}). \label{CV vj v C0}
\end{gather}
Indeed, since $B$ is a ball, there exists an extension operator $E:W^{2,r}(B) \rightarrow W^{2,r}(\C^n)$.
Let then $\varphi \in C^{\infty}_c(\C^n)$ be a cut-off function which is equal to $1$ in $B$ and $0$ outside an open set $\omega$ such that $\omega \supset \clos{B}$.
Thus, $\hat{v}=\varphi Ev \in W^{2,r}_0(\omega)$ and there exists a sequence $(\hat{v}_j)_j \subset C^{\infty}_c(\omega)$ such that $\hat{v}_j \to \hat{v}$ in $W^{2,r}(\omega)$.
We define $v_j$ as the extension of $\hat{v}_j$ by zero outside $\omega$, which then satisfies \eqref{CV vj v}.
The convergence \eqref{CV vj v C0} is consequence of \eqref{CV vj v} by the Sobolev embedding $W^{2,r}(B) \subset C^0(\clos{B})$ since $r>n$.

Let us now introduce the functions
$$
g_j=
\begin{cases}
4^n n! \det\left(D^2 v_j \right), & \text{ if } \ddc v_j \geq 0, \\
0 & \text{ otherwise, }
\end{cases}
\quad
g=
\begin{cases}
4^n n! \det\left(D^2 v \right), & \text{ if } \ddc v \geq 0, \\
0 & \text{ otherwise. }
\end{cases}
$$
Thanks to \eqref{CV vj v}, we have
\begin{equation}\label{CV gj g}
g_j \rightarrow g \text{ in } L^{r/n}(B).
\end{equation}
Since $g_j, g \geq 0$, from \cite{BT82} and \cite{Kol98} there exist $w_j,w \in \PSH(B) \cap C^0(\clos{B})$, respective unique solutions to
$$
\begin{cases}
(\ddc w_j)^n=g_j \dV, & \text{in } B, \\
w_j=v_j , & \text{on }  \partial B,
\end{cases}
\quad
\begin{cases}
(\ddc w)^n=g \dV, & \text{in } B, \\
w=v , & \text{on }  \partial B.
\end{cases}
$$
Moreover, Kolodziej's stability theorem (contained in the proof of \cite[Theorem 3]{Kol96}) states that
$$
\norm{w_j-w}_{L^{\infty}(B)}\leq \norm{v_j-v}_{L^{\infty}(\partial B)}+C(n,r)\norm{g_j-g}_{L^{\frac{r}{n}}(B)}^{\frac{1}{n}}.
$$
It follows that $w_j \rightarrow w$ in $C^0(\clos{B})$ thanks to \eqref{CV vj v C0} and \eqref{CV gj g}.

Since $(\ddc w_j)^n \geq (\ddc v_j)^n$ with $w_j=v_j$ on $\partial B$ and $v_j \in C^2(B)$, we can apply Theorem \ref{sibony} to obtain that
$$w_j\leq v_j \quad \text{ in } \clos{B}.$$

It follows that
\begin{align*}
w- v & = w-w_j+w_j-v_j+v_j-v \\
& \leq w-w_j +v_j-v.
\end{align*}
Passing to the limit, we deduce that
\begin{equation}\label{w less than v}
w-v \leq 0 \quad \text{ in } \clos{B}.
\end{equation}
To complete the proof, note that for $z\in B$, 
\begin{eqnarray*}
u(z)-v(z)&=& u(z)-w(z)+w(z)-v(z) \\
&\leq & u(z)-w(z) \quad \text{(by \eqref{w less than v})}, \\
&\leq & \max_{z\in \partial B} \, (u(z)-w(z)) \quad \text{(by Theorem \ref{compa BT} since $(\ddc u)^n \geq g \dV=(\ddc w)^n$)}, \\
&= & \max_{z\in \partial B} \, (u(z)-v(z)) \quad \text{(since $w=v$ on $\partial B$)}.
\end{eqnarray*}

\end{proof}

\section{Maximum principle for $W^{2,r}/L^p$-viscosity subsolutions}\label{sect max PP}

The starting point of any viscosity theory is the maximum principle.
Below we prove a version adapted to our setting.
Similar result for smoother functions can be found in \cite[Theorem 1]{RS64}.

\begin{theorem}\label{linmaxprin}
Let $\aij:\Omega \longrightarrow \C$ ($1 \leq i,j \leq n$) be such that the coefficient matrix $(a^{i\bar{j}}(z))_{1 \leq i,j \leq n}$ is Hermitian and nonnegative for a.e. $z \in \Omega$.
Assume that there exists $M>0$ such that, for a.e. $z \in \Omega$,
\begin{equation}\label{cond max PP}
\sum_{i=1}^n a^{i\bar{i}}(z) \geq M.
\end{equation}
Let $u \in C^0(\clos{\Omega})$ be a $W^{2,r}/L^p$-viscosity solution ($r,p \geq 1$) of
\begin{equation}\label{equ max PP}
\sum_{i,j=1}^n\aij(z) u_{i\bar{j}} \geq 0 \quad \mbox{ in } \Omega.
\end{equation}
Then, $\max_{\clos{\Omega}} u=\max_{\partial \Omega} u$.
\end{theorem}

\begin{proof}
Suppose on contrary that the value $\max_{\clos{\Om}}u$ is reached only in $\Omega$ and let 
$$K=\ens{z\in \clos{\Omega} \quad \middle| \quad u(z)=\max_{\clos{\Om}}u}.$$
Clearly, $K$ is compact and, by assumption, $K \subset \Omega$.
Our goal will be to construct a strictly concave polynomial barrier at (some) point of $K$.
This coupled with the properties of our linear operator would lead to a contradiction.

Assume first that $K$ is reduced to a point, or more generally that it has an isolated point $z_0$.
Let then $R>0$ be fixed such that $B_R(z_0) \subset \Omega$ and such that no other maximum lies in $B_R(z_0)$.
Set $h(z)=-\varepsilon\norm{z-z_0}^2+u(z_0)$.
For $\epsilon>0$ small enough, the maximum of the difference $u-h$ is necessarily achieved inside $B_R(z_0)$, say at some point $z_1$.
Then, for the smooth function $\varphi(z)=h(z)+\frac{\epsilon}{2}\norm{z-z_1}^2$, the difference $u-\varphi$ has a strict maximum in $B_R(z_0)$.
As $\varphi_{i\bar{j}}(z)=-\frac{\epsilon}{2} \delta_{ij}$, we have
$$
\sum_{i,j=1}^n\aij(z) \varphi_{i\bar{j}}(z)
=-\frac{\epsilon}{2} \sum_{i=1}^n a^{i\bar{i}}(z)
\leq -\frac{\epsilon}{2} M,
$$
and hence $\varphi$ is a valid test function, which contradicts the fact that $u$ is a $W^{2,r}/L^p$-viscosity solution of \eqref{equ max PP} (see Definition \ref{def visc sol}).

In general the set $K$ need not contain isolated points and we proceed differently.
Pick a point $x_0\in K$ and inflate the balls centered at $x_0$ up until the whole set $K$ is inside.
If $K$ is not reduced to a point, then $R=\max_{z \in K} \norm{z-x_0}$ is positive and we have $K\subset \clos{B_R(x_0)}$ and $K \cap \partial B_R(x_0) \neq \emptyset$.
Let then $z_0\in K \cap \partial B_R(x_0)$.
Rotating and shifting coordinates if necessary, we may assume that $x_0=(0,\cdots,0,R)$, $K\subset \ens{\norm{z'}^2+|z_n-R|^2\leq R^2}$ and $z_0=(0,\cdots,0)\in\partial K$ (observe that the assumptions of the theorem are invariant under such transformations).
Note that any point $z=(z',x_n+iy_n) \in K$ has to satisfy
$$x_n\geq\frac{y_n^2+\norm{z'}^2}{2R}.$$

Then fixing a constant $C>2R$, for any $\eta>0$  the set
\begin{multline*}
\hat{U}_\eta=\ens{z \in \C^n \quad \middle| \quad  \real z_n=-\eta, \quad (\imag z_n)^2+\norm{z'}^2\leq C\eta}
\\
\cup
\ens{z \in \C^n \quad \middle| \quad |\real z_n|\leq \eta, \quad (\imag z_n)^2+\norm{z'}^2=C\eta},
\end{multline*}
is disjoint from $K$.
Since $\hat{U}_\eta$ is compact, it follows that $\max_{\hat{U}_\eta} u<\max_{\clos{\Om}} u=u(0)$, so that there is a $\delta>0$, such that
\begin{equation}\label{dist to zero}
u|_{\hat{U}_\eta}<u(0)-2\delta.
\end{equation}

Since $0 \in K \subset \Omega$, there is a ball $B_r(0) \subset \Omega$ and for $\eta$ small enough so that $\eta^2+C\eta<r^2$, the domain
$$U_\eta=\ens{z \in \C^n \quad \middle|\quad \ |\real z_n|<\eta, \quad (\imag z_n)^2+\norm{z'}^2<C\eta},$$
is such that $\clos{U_\eta} \subset \Omega$.
Fix such an $\eta$ and set $U=U_\eta$.

Consider now the quadratic polynomial
$$h(z)=-\frac{\delta}{4C\eta}(\norm{z'}^2+(\imag z_n)^2)-\frac{\delta}{3\eta^2}(\real z_n-\eta)^2+\frac{\delta}{3}+u(0).$$

By continuity, $u-h$ has a maximum on $\clos{U}$.
Let us show that this value is necessarily reached in $U$.
On $\clos{U}$, we have
$$h(z)\geq -\frac{\delta}{4}-\frac{\delta}{3\eta^2}(\real z_n-\eta)^2+\frac{\delta}{3}+u(0).$$
On  $\hat{U}_\eta$, using \eqref{dist to zero} we obtain
$$h(z)\geq -\frac{5\delta}{4}+u(0)>u(z).$$
On the remaining part of $\partial U$, i.e. on the piece where $ \real z_n=\eta$, we simply use that $0 \in K$:
$$h(z)\geq \frac{\delta}{12}+u(0)>u(0)\geq u(z).$$
In summary, $u-h<0$ on $\partial U$.
Since $0 \in U$ with $u(0)-h(0)=0$, the maximum of $u-h$ has to be reached into $U$.
We also have $h_{i\bar{j}}(z)=-\varepsilon_i \delta_{ij}$ for some $\varepsilon_i>0$.

The conclusion is now as before: define $\varphi(z)=h(z)+\frac{\epsilon}{2}\norm{z-z_1}^2$, where $z_1$ is a point of maximum of $u-h$ in $U$, so that the difference $u-\varphi$ now has a strict maximum in $U$; taking $\epsilon<\epsilon_i$ and using $a^{i\bar{i}} \geq 0$, we have
$$
\sum_{i,j=1}^n\aij(z) \varphi_{i\bar{j}}(z)
=- \sum_{i=1}^n a^{i\bar{i}}(z)\left(\varepsilon_i-\epsilon\right)
\leq - M\min_{1 \leq i \leq n} \left(\varepsilon_i-\epsilon\right),
$$
and hence $\varphi$ is a valid test function, which contradicts the fact that $u$ is a $W^{2,r}/L^p$-viscosity solution of \eqref{equ max PP}.
\end{proof}

\section{Proof of the main result}\label{sect proof}

In this section we finally prove Theorem \ref{main thm}.
All along this section, fix a sufficiently large ball containing $\Om$. 
Without loss of generality, we assume that it is $B_1=B_1(0)$.
We then extend $g$ by zero in $B_1 \backslash \Omega$ (still denoted by $g$).

Let then $\rho \in C^0(\clos{B_1})$ be the corresponding $W^{2,r}/L^{p/n}$-viscosity solution of
$$F_{\MA}(\DD \rho) \geq (g_+)^n \quad \mbox{ in } B_1,$$
provided by Theorem \ref{thm rho} with $q=p/n$, whose hypotheses \eqref{cond rq} are satisfied thanks to our assumption \eqref{cond rp}.
It also satisfies $\rho=0$ on $\partial B_1$ and the Ko\l odziej's estimate
\begin{equation}\label{estim rho}
\sup_{B_1} \, (-\rho)\leq C(n,q) \norm{g_+}_{L^p(B_1)}.
\end{equation}

We will show that $-\rho$ is an upper barrier for $-u$ in $\Omega$.
The conclusion will then follow from \eqref{estim rho}.
Let $L_u$ be the linearization of $G$ about $D^2 u$:
$$L_u h= \sum_{i,j=1}^n G^{i \bar{j}}(z,(D^2 u)(z)) h_{i\bar{j}}.$$

First of all, we have the following lemma:

\begin{lemma}\label{w visc sol}
The function $\rho$ is a $W^{2,r}/L^p$-viscosity solution of 
$$L_u \rho \geq g_+
\quad \mbox{ in } B_1.$$
\end{lemma}

\begin{proof}
Assume not.
Then, there exist $\varphi\in W^{2,r}_{\loc}(B_1)$, $\varepsilon>0$ and an open subset $U \subset B_1$ such that, in $U$,
\begin{equation}\label{non visco}
L_u \varphi
\leq g_+ -\varepsilon,
\end{equation}
and $\rho-\varphi$ has a strict local maximum in $U$.
Since $\rho$ is a $W^{2,r}/L^{p/n}$-viscosity solution of $F_{\MA}(\DD \rho) \geq (g_+)^n$ in $B_1$, by very definition we obtain that, for every $\eta>0$, the set
$$V_\eta=\ens{z \in U \quad \middle| \quad F_{\MA}(\DD \varphi(z))> (g_+(z))^n-\eta},$$
has to be of positive measure.
In particular, note that $\DD\varphi \in \coneb_n$ a.e. in $V_\eta$ since $F_{\MA}$ is equal to $-\infty$ outside $\coneb_n$ by definition.
In this set $V_\eta$, we have
\begin{align*}
L_u \varphi
&\geq  G\left(z,\DD \varphi\right) \quad \text{(by \eqref{basicineq})}, \\
&\geq \left(\det (\DD \varphi)\right)^{\frac{1}{n}} \quad \text{(by \eqref{comparison with det} since $\DD\varphi \in \coneb_n$)}, \\
&\geq g_+-\eta^{\frac{1}{n}} \quad \text{(by definition of $V_\eta$ and subadditivity of $x \longmapsto x^{\frac{1}{n}}$)}.
\end{align*}
As $\varepsilon>0$ is fixed, by taking $\eta>0$ small enough we reach a contradiction with \eqref{non visco}.

\end{proof}

We can now prove our main result:

\begin{proof}[Proof of Theorem \ref{main thm}]
By homogeneity, we have
$$L_u u=G(z,\DD u).$$
Therefore,
$$L_u (-u)=-G(z,\DD u) \geq -g \geq -g_+.$$
From Lemma \ref{w visc sol} and Remark \ref{sum visc sol} we see that $\rho-u$ is then a $W^{2,r}/L^p$-viscosity solution to
$$L_u\left(\rho-u\right) \geq 0 \quad \mbox{ in } \Omega.$$

We can check that $\aij(z)=G^{i\bar{j}}(z,(\DD u)(z))$ satisfy the assumptions of Theorem \ref{linmaxprin}.
Indeed, the condition \eqref{cond max PP} follows from \eqref{basicineq} (with $B=\Id$) and the fact that $G(z,\Id) \geq 1$ (by \eqref{comparison with det}); the nonnegativity of $(a^{i\bar{j}}(z))_{1 \leq i,j \leq n}$ follows from the ellipticity of $A \mapsto G(z,A)$.
Therefore, using this maximum principle, we obtain that (recall that $\rho \leq 0$ in $\clos{B_1}$ and $u \geq 0$ on $\partial\Omega$)
$$\rho-u \leq 0 \quad \mbox{ in } \Omega.$$
The desired estimate \eqref{max estim nonlin} then follows from Ko\l odziej's estimate \eqref{estim rho} of $\rho$.
\end{proof}

\section*{Acknowledgements}

The first and third named authors were partially supported by the Polish National Science Centre Grant 2017/26/E/ST1/00955.
The second named author was partially supported by the  Polish National Science Centre Grant 2017/27/B/ST1/01145.

\bibliographystyle{amsalpha}
\bibliography{biblio}

\end{document}